\documentclass{amsart}
 
\usepackage[all]{xy}
\usepackage{graphicx}
\usepackage{amsmath,amsxtra,amssymb,latexsym, amscd,amsthm, enumerate}

\usepackage{tikz}
\usepackage{ytableau}

\definecolor{red}{RGB}{255, 0, 0}
\definecolor{orange}{RGB}{255, 165, 0}
\definecolor{yellow}{RGB}{255, 255, 0}
\definecolor{chartreuse}{RGB}{127, 255, 0}
\definecolor{green}{RGB}{0, 255, 0}
\definecolor{springgreen}{RGB}{0, 255, 127}
\definecolor{cyan}{RGB}{0, 255, 255}
\definecolor{azure}{RGB}{0, 127, 255}
\definecolor{blue}{RGB}{0, 0, 255}
\definecolor{violet}{RGB}{127, 0, 255}

 \usepackage{indentfirst}
\usepackage[mathscr]{eucal}
  \usepackage[pagebackref=true]{hyperref}

\newtheorem{thm}{Theorem}[section]

\newtheorem{lem}[thm]{Lemma}

\theoremstyle{definition}

\newtheorem{defn}[thm]{Definition}
\newtheorem{exam}[thm]{Example}

\newtheorem{rem}[thm]{Remark}

\numberwithin{equation}{section}

\DeclareMathOperator{\NN}{\mathbb {N}}
\DeclareMathOperator{\ZZ}{\mathbb {Z}}

\DeclareMathOperator{\height}{height}

\DeclareMathOperator{\depth}{depth}

\def\a {\mathbf a}
\def\b {\mathbf b}

\def\m {\mathfrak m}

\def\k {\mathrm{k}}

\def\ww {\mathbf{w}}

\usepackage{color}
\newcommand{\red}{\textcolor[rgb]{1.00,0.00,0.00}}
\begin{document}

\title{Unmixed and sequentially Cohen-Macaulay skew tableau ideals}

\author{Do Trong Hoang}
\address{Faculty of  Mathematics and Informatics, Hanoi University of Science and Technology, 1 Dai Co Viet, Hai Ba Trung, Hanoi, Vietnam.}

\email{hoang.dotrong@hust.edu.vn}

\author{Thanh Vu}
\address{Institute of Mathematics, VAST, 18 Hoang Quoc Viet, Hanoi, Vietnam}
\email{vuqthanh@gmail.com}


\date{}

\dedicatory{Dedicated to Professor Uwe Nagel on the occasion of his 60th birthday}
\commby{}

\begin{abstract}
   We associate a {\it skew tableau ideal} to each filling of a skew Ferrers diagram with positive integers. We classify all unmixed and sequentially Cohen-Macaulay skew tableau ideals. Consequently, we classify all Cohen-Macaulay, Buchsbaum, and generalized Cohen-Macaulay skew tableau ideals.
\end{abstract}

\maketitle
\section{Introduction} Motivated by the work of Corso and Nagel \cite{CN}, in \cite{HV}, we studied ideals associated with arbitrary fillings of positive integers in a Young diagram, known as tableau ideals. In particular, we proved that the tableau ideal $I(T)$ is Cohen-Macaulay if and only if $T$ has the shape $(n,n-1,\ldots,1)$ for some $n$ and the filling of $T$ is weakly increasing along both rows and columns of $T$. Skew Young tableaux, a natural generalization of Young tableaux, play a fundamental role in representation theory and symmetric functions \cite{Ful}. It is natural to extend our studies to skew tableau ideals. In this study, we classify both unmixed and sequentially Cohen-Macaulay skew tableau ideals. Our findings indicate that a skew shape is unmixed if and only if its corner blocks are square. Additionally, we demonstrate that a Ferrers diagram is sequentially Cohen-Macaulay if and only if its rectangular corner blocks have $\min(\text{width, height}) = 1$. Moreover, a sequentially Cohen-Macaulay skew shape can be formed by incrementally adding such rectangular blocks to sequentially Cohen-Macaulay diagrams. See Section \ref{sec_scm_skew_Ferrers} for a more precise statement. Sequentially Cohen-Macaulay ideals possess many desirable properties similar to those of Cohen-Macaulay ideals. In particular, their depth and regularity are much better understood. Let us now introduce these concepts in greater detail.

A Young diagram, or Ferrers diagram, is a finite collection of boxes, or cells, arranged in left-justified rows, with the row lengths in non-increasing order. The numbers of boxes in each row give a partition $\lambda$ of the total number of boxes in the diagram. A skew shape is a pair of partitions $(\lambda, \mu)$ such that the Young diagram of $\lambda$ contains the Young diagram of $\mu$. The skew diagram of a skew shape $\lambda/\mu$ is the set of boxes that belong to the diagram of $\lambda$ but not to that of $\mu$. For each filling $Y$ of the skew diagram $\lambda/\mu$ whose the values at the box in the $i$th row and $j$th column is $w(i,j)$, we associated with a skew tableau ideal $I(Y) \subseteq S = \k[x_1,\ldots,x_n,y_1,\ldots,y_m]$, where $n$ is the length of $\lambda$ and $m = \lambda_1$ as follows
$$I(Y) = ( (x_iy_j)^{w(i,j)}\mid 1\le i\le n, \mu_i+1\le j \le \lambda_i).$$	
The radical of $I(Y)$, denoted by $I_{\lambda/\mu}$, is called the skew Ferrers ideal. Skew Ferrers ideals also arise as the initial ideals of the closed binomial edge ideals and have been studied in \cite{dAH, H, P}. 

Letting $\mu = 0$ be the partition of $0$, we obtain the tableau ideals corresponding to the shape $\lambda$. To study the sequentially Cohen-Macaulay property of skew tableau ideals, we first need to classify all skew shapes $\lambda/\mu$ such that $I_{\lambda/\mu}$ is sequentially Cohen-Macaulay. We now recall the definition of sequentially Cohen-Macaulay modules over $S$ introduced by Stanley \cite{S}.
\begin{defn} Let $M$ be a graded module over $S$. We say that $M$ is
	sequentially Cohen-Macaulay if there exists a filtration
	$$0 = M_0\subset M_1 \subset \cdots \subset M_r = M$$
	of $M$ by graded $S$-modules such that $\dim (M_i/M_{i-1}) <\dim (M_{i+1}/M_i)$
	for all $i$, where $\dim$ denotes Krull dimension, and $M_i/M_{i-1}$ is Cohen-Macaulay for all $i$. An ideal $J$ is said to be sequentially Cohen-Macaulay if $S/J$ is a sequentially Cohen-Macaulay $S$-module. A skew shape $\lambda/\mu$ is called sequentially Cohen-Macaulay if the corresponding skew Ferrers ideal $I_{\lambda/\mu}$ is.
\end{defn}

We first classify all sequentially Cohen-Macaulay Ferrers ideals. By convention, for a partition $\lambda = (\lambda_1, \ldots, \lambda_n)$ we set $\lambda_j = 0$ for all $j > n$.

\begin{defn} We say that a partition $\lambda$ is saturated if either $\lambda$ is strictly decreasing or $\{\lambda_i,\lambda_i-1, \ldots, 1\} \subseteq \{\lambda_1, \ldots, \lambda_n\}$ where $i$ is the smallest index such that $\lambda_i = \lambda_{i+1}$.
\end{defn}

\begin{thm}\label{thm_scm_Ferrers} Let $\lambda = (\lambda_1,\ldots, \lambda_n)$ be a partition. Then $I_\lambda$ is sequentially Cohen-Macaulay if and only if $\lambda$ is saturated.
\end{thm}

We then show that the sequentially Cohen-Macaulay property of skew Ferrers ideals can be determined by the sequentially Cohen-Macaulay property of its properly sub skew Ferrers ideals recursively. We denote by $G_{\lambda/\mu}$ the skew Ferrers graph associated with the skew Ferrers ideal $I_{\lambda/\mu}$.

\begin{thm}\label{thm_scm_skew_Ferrers} Let $\lambda/\mu$ be a skew shape. Let $\lambda', \mu'$ be the conjugate partitions of $\lambda,\mu$ respectively. Then $G_{\lambda/\mu}$ is sequentially Cohen-Macaulay if and only if one of the following conditions holds.
\begin{enumerate}
    \item $\lambda_1 > \lambda_2$. Then $G_{\lambda/\mu} \backslash x_1$ and $G_{\lambda/\mu} \backslash N[x_1]$ are sequentially Cohen-Macaulay.
    \item $\mu_1+1 = \lambda_1$. Then $G_{\lambda/\mu} \backslash y_{m}$ and $G_{\lambda/\mu} \backslash N[y_m]$ are sequentially Cohen-Macaulay.
    \item $\lambda_1' > \lambda_2'$. Then $G_{\lambda/\mu} \backslash y_1$ and $G_{\lambda/\mu} \backslash N[y_1]$ are sequentially Cohen-Macaulay.
    \item $\mu_1' + 1  = \lambda_1'$. Then $G_{\lambda/\mu} \backslash x_n$ and $G_{\lambda/\mu} \backslash N[x_n]$ are sequentially Cohen-Macaulay.
\end{enumerate}
where $G \backslash V$ denotes the induced subgraph of $G$ on $V(G) \backslash V$ and $N[v]$ denotes the closed neighborhood of $v$ in $G$.
\end{thm}

Since a sequentially Cohen-Macaulay ideal is Cohen-Macaulay if and only if it is unmixed, we now classify unmixed skew Ferrers ideals. We introduce the following notation that describes the unmixed Ferrers diagrams.

\begin{defn}  Let $\lambda = (\lambda_1,\ldots,\lambda_n)$ be a partition and $\lambda'$ be the conjugate partition of $\lambda$. The rectangular regions of the Ferrers diagram corresponding to $\lambda$, cut out by the horizontal lines at $x = \lambda_i'$ and vertical lines at $y = \lambda_j$ are called blocks of $\lambda$. The corner blocks of $\lambda$ are those that do not have adjacent blocks that lie to the right or below them. The extremal blocks of $\lambda$ are the top-right and bottom-left corner blocks. The size of a block is the minimum value of its depth and width.
\end{defn}

Using this terminology, the result of Corso and Nagel \cite[Corollary 2.7]{CN} can be rephrased as: $I_\lambda$ is unmixed if and only if all corner blocks of $\lambda$ are square. In the following, a skew shape $\lambda/\mu$ is called unmixed, Cohen-Macaulay, generalized Cohen-Macaulay, Buchsbaum if the corresponding skew Ferrers ideal $I_{\lambda/\mu}$ is.

\begin{defn} The diagram of an unmixed partition $\lambda$ is called an upper triangular prime unmixed shape. Its reflection along the main diagonal is called a lower triangular prime unmixed shape. A prime unmixed shape is either an upper triangular or a lower triangular prime unmixed shape.

Two prime unmixed shapes sharing an extremal block are said to be glued together along this extremal block. 
\end{defn}

\begin{thm}\label{thm_unmixed_skew_Ferrers} Let $\lambda/\mu$ be a skew shape. Assume that $G_{\lambda/\mu}$ is connected. Then $I_{\lambda/\mu}$ is unmixed if and only if there exists a sequence of prime unmixed shapes $B_1,\ldots, B_t$ such that $\lambda/\mu = B_1 \cup \cdots \cup B_t$, $B_i, B_{i+1}$ are of alternating type (upper/lower triangular) and $B_i \cap B_{i+1}$ is the extremal block of $B_i$ and $B_{i+1}$ for all $i = 1, \ldots, t-1$.
\end{thm}
Consequently, we deduce a characterization of (generalized) Cohen-Macaulay and Buchsbaum skew Ferrers ideals. See Section \ref{sec_scm_skew_Ferrers} for more details.
\begin{exam} An unmixed skew shape is obtained by gluing a sequence of alternating type prime unmixed shapes. In the following figure, boxes with the same labels form a block. The first prime unmixed shape consists of the blocks $a$, $b$, $c$, $d$, $e$, $f$. The second prime unmixed shape consists of the blocks $f$, $g$, and $h$. The third prime unmixed shape consists of the blocks $h$, $i$, $j$. The first and second prime shapes are glued along their extremal block $f$. The second and third prime shapes are glued along their extremal block $h$.
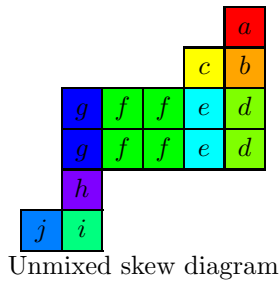
 \begin{center}  
     \begin{tabular}{ccc}
     $\begin{ytableau}
       \none & \none & \none & \none  & \none & *(red)a \\
       \none & \none & \none & \none & *(yellow)c & *(orange)b \\
     \none    & *(blue)g & *(green)f & *(green)f & *(cyan)e & *(chartreuse)d \\
     \none    & *(blue)g & *(green)f & *(green)f & *(cyan)e & *(chartreuse)d \\
    \none    & *(violet)h   \\
    *(azure)j    & *(springgreen)i   
   \end{ytableau}$\\[15pt]
   \text{Unmixed skew diagram}
   \end{tabular}
   \end{center}
\end{exam}

We now come to skew tableau ideals. Skew tableau ideals are the edge ideals of the edge weighted graphs associated with the skew Ferrers graphs. For the unmixed property, we prove
\begin{thm}\label{thm_unmixed_skew_tableau} Let $Y$ be a filling of positive integers in a skew shape $\lambda/\mu$. Assume that $G_{\lambda/\mu}$ is connected. Then $I(Y)$ is unmixed if and only if $I_{\lambda/\mu}$ is unmixed and the following conditions hold:
\begin{enumerate}
    \item the filling of $Y$ is constant in each blocks of $\lambda/\mu$,
    \item the filling of $Y$ on upper triangular prime shapes of $\lambda/\mu$ is weakly increasing along both rows and columns,
    \item the filling of $Y$ on lower triangular prime shapes of $\lambda/\mu$ are weakly decreasing along both rows and columns.
\end{enumerate}
\end{thm}
Finally, we come to classify sequentially Cohen-Macaulay skew tableau ideals. Diem, Minh, and Vu \cite{DMV} recently proved that the edge ideals of edge weighted graphs $I(G,\ww)$ is sequentially Cohen-Macaulay for all weight functions $\ww$ if and only if $G$ does not have an induced cycle $C_k$ for $k \neq 3,5$. Since most skew Ferrers graphs have an induced $4$-cycle, not all skew tableau ideals associated with a sequentially Cohen-Macaulay skew shape are sequentially Cohen-Macaulay. 

\begin{thm}\label{thm_scm_skew_tableau} Let $Y$ be a filling of positive integers in a skew shape $\lambda/\mu$. Assume that $G_{\lambda/\mu}$ is connected. Let $\lambda',\mu'$ be the conjugate partitions of $\lambda,\mu$ respectively. Then $I(Y)$ is sequentially Cohen-Macaulay if and only if one of the following conditions holds.
\begin{enumerate}
    \item $\lambda_1 > \lambda_2$. Let $\omega_1 = \max \{w(1,j) \mid j > \lambda_2\}$. Let $Y_1$ be the skew tableau obtained by deleting the first row of $Y$ and $Y_2$ be the skew tableau obtained by deleting all the $t$th column of $Y$ such that $w(1,t) \le \omega_1$. Then $Y_1$ and $Y_2$ are sequentially Cohen-Macaulay.
    \item $\mu_1 + 1 = \lambda_1$. Let $p$ be the largest index such that $\mu_p = \mu_1$. Let $\delta_m = \max \{ w(i,m) \mid i \ge p\}$. Let $Y_1$ be the skew tableau obtained by deleting the last column of $Y$ and $Y_2$ be the skew tableau obtained by deleting all the $s$th row of $Y$ such that $w(s,m) \le \delta_m$. Then $Y_1$ and $Y_2$ are sequentially Cohen-Macaulay.
    \item $\lambda_1' > \lambda_2'$. Let $\delta_1 = \max \{ w(i,1) \mid i \ge \lambda_2'+1\}$. Let $Y_1$ be the skew tableau obtained by deleting the first column of $Y$ and $Y_2$ be the skew tableau obtained by deleting all the $s$th row of $Y$ such that $w(s,1) \le \delta_1$. Then $Y_1$ and $Y_2$ are sequentially Cohen-Macaulay.
    \item $\mu_1' + 1 = \lambda_1'$. Let $q$ be the largest index such that $\mu_q' = \mu_1'$. Let $\omega_n = \max \{ w(n,j) \mid j \le q\}$. Let $Y_1$ be the skew tableau obtained by deleting the last row of $Y$ and $Y_2$ be the skew tableau obtained by deleting all the $t$th column of $Y$ such that $w(n,t) \le \omega_n$. Then $Y_1$ and $Y_2$ are sequentially Cohen-Macaulay.
\end{enumerate}
\end{thm}

\begin{exam}
    Let  $\lambda=(5,4,4)$ and $\mu=(2,1)$. Then $I_{\lambda/\mu}$ is sequentially Cohen-Macaulay. Consider the filling $Y$ of $\lambda/\mu$ as in the following. 
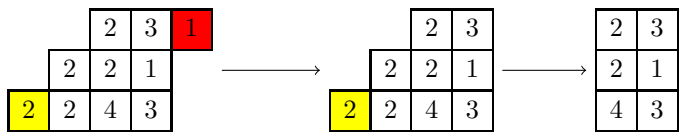
\begin{center}
     
\begin{tikzpicture}

\node (A) {
\begin{ytableau}
\none     & \none & *(white) 2& *(white)3 & *(red)1  \\
     \none   & *(white)2   & *(white)2   & *(white)1     \\
    *(yellow)2    & *(white)2 & *(white)4& *(white)3
\end{ytableau}
};
\node[right of=A, xshift=3cm] (B)  {
\begin{ytableau}
\none     & \none & *(white) 2& *(white)3  \\
     \none   & *(white)2   & *(white)2   & *(white)1     \\
    *(yellow)2    & *(white)2 & *(white)4& *(white)3
\end{ytableau}
};
\node[right of=B, xshift=2cm]  (C) {
\begin{ytableau}
*(white) 2& *(white)3  \\
 *(white)2   & *(white)1     \\
 *(white)4& *(white)3
\end{ytableau}
};

\draw[->] (A) -- (B);
\draw[->] (B) -- (C);

\end{tikzpicture}

\end{center}    
    By Theorem \ref{thm_scm_skew_tableau}, the sequentially Cohen-Macaulay property is preserved in the sequence of tableau ideals as in the figure. Since the last one corresponds to a non sequentially Cohen-Macaulay shape, we deduce that $I(Y)$ is not sequentially Cohen-Macaulay. Note that, if we change the weight of the yellow box to $3$ we would have a sequentially Cohen-Macaulay skew tableau ideal.
\end{exam}
In Section \ref{sec_scm_skew_Ferrers}, we recall necessary notations and prove Theorem \ref{thm_scm_Ferrers} and Theorem \ref{thm_scm_skew_Ferrers}. We then deduce a characterization of (generalized) Cohen-Macaulay and Buchsbaum skew Ferrers ideals. In Section \ref{sec_skew_tableau}, we prove Theorem \ref{thm_unmixed_skew_tableau} and Theorem \ref{thm_scm_skew_tableau}. We then  characterize (generalized) Cohen-Macaulay and Buchsbaum skew tableau ideals.

\section{Sequentially Cohen-Macaulay skew Ferrers ideals}\label{sec_scm_skew_Ferrers}
In this section, we classify all sequentially Cohen-Macaulay skew Ferrers ideals, and then all (generalized) Cohen-Macaulay and Buchsbaum skew Ferrers ideals. We first recall the definition of skew Ferrers ideals and skew Ferrers graphs. 

\begin{defn} Let $\lambda = (\lambda_1, \ldots, \lambda_n)$ and $\mu = (\mu_1,\ldots,\mu_n)$ be partitions such that $\lambda_i > \mu_i \ge 0$ for all $i = 1, \ldots, n$ and $\mu_n = 0$. Set $\lambda_1 = m$. The skew Ferrers ideal associated with the skew shape $\lambda/\mu$ is defined as
$$I_{\lambda/\mu}= (x_iy_j\mid 1\le i\le n, \mu_i+1\le j \le \lambda_i)  \subset S = \k[x_1,\ldots,x_n, y_1,\ldots,y_m].$$	
The skew Ferrers graph $G_{\lambda/\mu}$ is the bipartite graph on the vertex set $X \cup Y$ where $X = \{x_1, \ldots, x_n\}$, $Y = \{y_1, \ldots,y_m\}$ and $\{x_i,y_j\}$ is an edge of $G_{\lambda/\mu}$ if and only if $\mu_i + 1 \le j \le \lambda_i$.

When $\mu = 0$, the skew Ferrers ideal and skew Ferrers graph associated with $\lambda/\mu$ are called the Ferrers ideal and Ferrers graph associated with $\lambda$ and denoted by $I_\lambda$, $G_\lambda$ respectively.
\end{defn}

Changing the roles of $x$ and $y$ corresponds to taking the conjugate shape $\lambda'/\mu'$.

\begin{defn} Let $\lambda = (\lambda_1, \ldots, \lambda_n)$ be a partition. The conjugate of $\lambda$ is the partition $\lambda' = (\lambda'_1, \ldots, \lambda'_m)$ where $m = \lambda_1$ and $\lambda'_j$ is the number of boxes on the $j$th column of $\lambda$ for all $j = 1, \ldots,m$.    
\end{defn}
\begin{exam} Let $\lambda= (5,5,4)$. Then its conjugate partition is $\lambda'=(3, 3, 3,3,2)$.
\end{exam}
We make the following convention throughout the paper. For a skew shape $\lambda/\mu$ with $\lambda  = (\lambda_1,\ldots,\lambda_n)$ we assume that $\mu_n = 0$ and set $\lambda_0 = \lambda_1 + 1 = \mu_0 + 1$, $\lambda_j = 0$ for all $j > n$ and $\mu_j = -1$ for all $j > n$. For a graph $G$ on the vertex set $V$ and a vertex $v \in V$, we denote by $N_G(v)$ the set of neighbors of $v$ in $G$. The closed neighborhood of $v$ in $G$ is $N_G[v] = N_G(v) \cup \{v\}$. We have 
\begin{lem} \label{connected}
  The graph $G_{\lambda/\mu}$ is connected if and only if $\mu_{i-1} \le   \lambda_{i}-1$ for all $2\le i\le n$. 
\end{lem}
\begin{proof}  Assume that there exists $2\le i\le n$  such that $\mu_{i-1} \ge  \lambda_{i}$. Let $G_1$ and $G_2$ be two induced subgraphs of $G_{\lambda/\mu}$ on the vertex set $\{x_i,\ldots,x_n\}\cup \{y_1,\ldots,y_{\lambda_i}\}$ and $\{x_1,\ldots,x_{i-1}\}\cup \{y_{\mu_{i-1}+1},\ldots,y_{m}\}$, respectively. Then $G_{\lambda/\mu}$  is disjoint union of $G_1$ and $G_2$. Thus $G_{\lambda/\mu}$  is not connected. 
 
Now assume that $\mu_{i-1} \le \lambda_i - 1$ for all $2 \le i \le n$. Under this assumption,  $G$ is the union of the paths $P_i: x_iy_{\mu_i+1}, \ldots, x_iy_{\lambda_i}$ and the path $P_i$ intersects $P_{i+1}$ non trivially for all $i = 1,\ldots, n-1$. Hence, $G$ is connected.  
\end{proof}

Let $I = J + K$ be the mixed sum of two nonzero homogeneous ideals living in different polynomial rings. Let $P$ be one of the following properties: sequentially Cohen-Macaulay, Cohen-Macaulay, Buchsbaum, generalized Cohen-Macaulay, unmixed. It is well known that $I$ has property $P$ if and only if both $J$ and $K$ have property $P$. Hence, we may assume that $G_{\lambda/\mu}$ is connected throughout the rest of the paper. To establish the sequentially Cohen-Macaulay property of skew tableau ideals, we utilize the notion of associated radicals of monomial ideals.

\begin{defn} Let $I$ be a monomial ideal in a polynomial ring $S$. A monomial ideal of the form $\sqrt{I:u}$ for a monomial $u \notin I$ is called an associated radical of $I$.
\end{defn}

\begin{lem} \label{lem_ass_radical_transfer} Let $I$ be a monomial ideal. Let $J = \sqrt{I:u}$ be an associated radical of $I$ and $K = J : v$ be an associated radical of $J$. Then $K$ is an associated radical of $I$.    
\end{lem}
\begin{proof} Let $s$ be the maximal exponent of a variable that appears in a minimal generator of $I$. We claim that $K = \sqrt{I:(uv^s)}$. Let $f$ be a squarefree monomial in $\sqrt{I:(uv^s)}$. We have $f^s \in I : (uv^s)$. In particular, $(fv) \in \sqrt{I:u}$. Hence, $f \in K = J : v $. Conversely, let $f$ be a squarefree monomial in $K = J : v$. Then $f v \in J$. Thus, $(fv)^s \in I:u$. Hence, $f \in \sqrt{I:(uv^s)}$. The conclusion follows.    
\end{proof}

We will use the following results to establish the sequentially Cohen-Macaulay property of skew tableau ideals.

\begin{lem}\label{lem_scm_radical} {\rm \cite[Proposition 2.23]{JS}} Let $I$ be a monomial ideal. Then $I$ is sequentially Cohen-Macaulay if and only if all associated radicals  of $I$  are sequentially Cohen-Macaulay.  
\end{lem}

\begin{defn} A graph $G$ is vertex decomposable if $G$ is a totally disconnected graph, i.e., $G$ has no edges or if there exists a vertex $v$ of $G$ such that 
\begin{enumerate}
    \item $G \backslash v$ and $G \backslash N[v]$ are both vertex decomposable, and
    \item no independent set in $G \backslash N[v]$ is a maximal independent set in $G\backslash v.$
\end{enumerate}
A vertex $v$ satisfying the latter two conditions is called a shredding vertex of $G$.
\end{defn}
We refer to \cite{V} and \cite{W} for more details about vertex decomposable graphs.

\begin{thm}\label{thm_vertex_decomposable} {\rm \cite[Theorem  2.10]{V}} Let $G$ be a bipartite graph. Then $I(G)$ is sequentially Cohen-Macaulay if and only if $G$ is vertex decomposable.     
\end{thm}

\begin{lem}\label{lem_non_vertex_decomp_complete_bipartite} Let $G = K_{U,V}$ is a complete bipartite graph. Assume that $U, V$ have at least $2$ vertices. Then $G$ is not vertex decomposable. 
\end{lem}
\begin{proof} Assume that $U = \{x_1,\ldots,x_n\}$ and $V = \{y_1,\ldots,y_m\}$ with $n,m\ge 2$. Assume by contradiction that $G$ is vertex decomposable. By definition, we may assume that $x_1$ is a shredding vertex of $G$. Then $G\backslash N_G[x_1]$ is the totally disconnected graph on $\{x_2,\ldots,x_n\}$. The set $\{x_2, \ldots,x_n\}$ is also a maximal independent set in $G\backslash x_1$, a contradiction. The conclusion follows.    
\end{proof}

We now classify sequentially Cohen-Macaulay Ferrers ideals. We first show that the condition of saturation of a partition is preserved when taking the conjugate. 

\begin{lem}\label{lem_conjugate} Let $\lambda$ be a partition and $\mu$ be its conjugate partition. Assume that $\lambda$ is saturated. Then $\mu$ is saturated.
\end{lem}
\begin{proof} We assume that $n$ is the length of $\lambda$ and $\lambda_1 = m$. We have the following simple observations.

\begin{enumerate}
    \item If $\lambda_1 > \cdots > \lambda_i$ then $\{i-1,\ldots,1\} \subseteq \{\mu_1,\ldots,\mu_m\}$.
    \item If $\{p,p-1,\ldots,1\} \subseteq \{\lambda_1,\ldots,\lambda_n\}$ then $\mu_1 > \cdots > \mu_{p+1}.$
\end{enumerate}
The conclusion then follows from the definition of saturation.
\end{proof}

\begin{lem}\label{lem_nec_Ferrers} Assume that $\lambda$ is not saturated. Then $I_\lambda$ is not sequentially Cohen-Macaulay.  
\end{lem}
\begin{proof} By definition, there exist indices $1 \le i < j \le n$ such that $i$ is the smallest index for which  $\lambda_i = \cdots = \lambda_j \ge  \lambda_{j+1} +2$. Let $p = \lambda_i$. We have 
$$I_{\lambda} : x_{j+1} y_{p+1} = (x_ly_k \mid i\le l \le j, \lambda_{j+1} +1 \le k \le \lambda_j) + (x_1,\ldots,x_{i-1},y_1,\ldots,y_q)$$
where $q = \lambda_{j+1}$. The conclusion follows from Theorem \ref{thm_vertex_decomposable}, Lemma \ref{lem_scm_radical}, and Lemma \ref{lem_non_vertex_decomp_complete_bipartite}.    
\end{proof}
\begin{proof}[Proof of Theorem \ref{thm_scm_Ferrers}] By Theorem \ref{thm_vertex_decomposable} and Lemma \ref{lem_nec_Ferrers}, it suffices to prove that if $\lambda$ is saturated, then $I_{\lambda}$ is vertex decomposable. We prove by induction on the size of $\lambda$. If $\lambda$ has only one row, the conclusion is obvious. Now, assume that $\lambda$ has at least two rows. There are two cases.

\smallskip 
\noindent{\textbf{Case 1.}} $\lambda_1 > \lambda_2$. We will prove that $x_1$ is a shredding vertex of $I_\lambda$. Let $G$ be the underlying graph associated with $I_\lambda$. We have $G\backslash x_1$ corresponds to the Ferrers ideals of shape $(\lambda_2, \ldots,\lambda_n)$ and $G\backslash N_G[x_1]$ is the totally disconnected graph on the vertex set $\{x_2, \ldots, x_n\}$. By induction, they are both vertex decomposable. Let $T = \{x_2, \ldots, x_n\}$ be the unique maximal independent set of $G \backslash N[x_1]$. Since $\lambda_1 > \lambda_2$, we have $T \cup \{y_m\}$ is an independent set of $G \backslash x_1$ where $m = \lambda_1$. 

\smallskip
\noindent{\textbf{Case 2.}} $\lambda_1 = \lambda_2$. By assumption, we must have $\{\lambda_1, \ldots, \lambda_n\} = \{\lambda_1,\lambda_1-1, \ldots, 1\}$. In particular, $\lambda_n = 1$. Let $\mu$ be the conjugate partition of $\lambda$. Then $I_\mu$ is isomorphic to $I_\lambda$. By Lemma \ref{lem_conjugate}, $\mu$ satisfies the condition of the theorem. Furthermore, $\mu_1 > \mu_2$. Since $I_\lambda \cong I_\mu$, the conclusion follows from Case 1.

That concludes the proof of the theorem.
\end{proof}

\begin{exam} Let $\lambda = (3,3,2,1)$. Then $I_\lambda$ is sequentially Cohen-Macaulay, as $\{3,3,2,1\} = \{3,2,1\}$. Note that $I_\lambda$ is not Cohen-Macaulay. Let $\mu = (4,3,3,1)$. Then $I_\mu$ is not sequentially Cohen-Macaulay, as $\{3,3,1\} = \{ 3,1\} \neq \{3,2,1\}$.
    \begin{center}
\begin{tabular}{ccc}
$\begin{ytableau}
*(white)  & *(white)  & *(white)  \\
*(white)  & *(white)  & *(white)    \\
*(white)  & *(white)      \\
*(white)  
\end{ytableau}$   
&\qquad\qquad&
$\begin{ytableau}
*(white)  & *(white)  & *(white) & *(white) \\
*(white)  & *(white)  & *(white)    \\
*(white)  & *(white)  & *(white)    \\
*(white)  
\end{ytableau}$\\[10pt]
 \text{sequentially Cohen-Macaulay} & & \text{NON sequentially Cohen-Macaulay}  
\end{tabular}
 \end{center}
\end{exam}

We now consider skew Ferrers ideals. First, we have a simple remark.

\begin{rem}\label{rem_flip} If $\lambda_1 = \cdots = \lambda_n = m$ then $I_{\lambda/\mu} \cong I_{\lambda'}$ where $\lambda' = (m-\mu_n,m-\mu_{n-1}, \cdots, m - \mu_1)$. Hence, if $\lambda = (m,\ldots,m)$ then we may use Theorem \ref{thm_scm_Ferrers} directly.     
\end{rem}

\begin{lem}\label{lem_shape_condition} Let $\lambda/\mu$ be a skew diagram and $\lambda'/\mu'$ be the conjugate skew diagram. Assume that $I_{\lambda/\mu}$ is sequentially Cohen-Macaulay. Then at least one of the following conditions must hold: 
\begin{enumerate}
    \item $\lambda_1 > \lambda_2$,
    \item $\mu_1 + 1 = \lambda_1$,
    \item $\lambda_1' > \lambda_2'$,
    \item $\mu_1' + 1 = \lambda_1'$.
\end{enumerate}
\end{lem}
\begin{proof} We prove by induction on $m + n$. The base case $n = 2$ is clear. Now assume that $n \ge 3$ and that $I_{\lambda/\mu}$ is sequentially Cohen-Macaulay. Assume by contradiction that $\lambda_1 = \lambda_2$, $\mu_1 + 1 < \lambda_1$, $\lambda_1' = \lambda_2'$ and $\mu_1' + 1 < \lambda_1'$. Since $I_{\lambda/\mu} \cong I_{\lambda'/\mu'}$ and by induction, we have $\lambda_1 = m \ge 3$. By Theorem \ref{thm_vertex_decomposable} and symmetry, we may assume that $y_j$ is a shredding vertex of $G = G_{\lambda/\mu}$ for some $j$. Since $G \backslash y_j$ corresponds to removing the $j$th column of $\lambda/\mu$. By induction, we deduce that either $j = 1, 2$ or $j = m-1,m$. Also, by symmetry, we may assume that $j = 1$ or $j = 2$. 

If $\mu = 0$, i.e., we have a Ferrers ideal, the conclusion follows from Theorem \ref{thm_scm_Ferrers}. Thus, we may assume that $\mu_1 > 0$. 

\smallskip 
\noindent{\textbf{Case 1.}} $j = 1$. We have $T = \{y_2, \ldots, y_m\}$ is a maximal independent set of $G \backslash N[y_1]$. For any $i = 1, \ldots, n$ we have $\{x_i,y_{\ell}\} \in E(G_{\lambda/\mu})$ for some $\ell \ge 2$. In other words, $T$ is a maximal independent set of $G \backslash y_1$, a contradiction. 

\smallskip 
\noindent{\textbf{Case 2.}} $j = 2$. We have $T = \{y_1, y_3, \ldots, y_m\}$ is a maximal independent set of $G \backslash N[y_2]$, which is also a maximal independent set of $G \backslash y_2$.

That concludes the proof of the lemma.
\end{proof}

\begin{proof}[Proof of Theorem \ref{thm_scm_skew_Ferrers}] First, assume that $I_{\lambda/\mu}$ is sequentially Cohen-Macaulay. By Lemma \ref{lem_shape_condition}, we deduce that either $\lambda_1 > \lambda_2$ or $\mu_1 + 1 = \lambda_1$ or $\lambda_1' > \lambda_2'$, or $\mu_1' + 1 = \lambda_1'$. Now, assume that $\lambda_1 > \lambda_2$. We have $I_{\lambda/\mu} : y_{\lambda_1} = I(G \backslash x_1) + (x_1)$ and $I_{\lambda/\mu} : x_1 = I ( G\backslash N[x_1]) + (y_j \mid \mu_1 + 1 \le j \le \lambda_1)$. By Lemma \ref{lem_scm_radical}, we deduce that $G \backslash x_1$ and $G\backslash N[x_1]$ are sequentially Cohen-Macaulay. The other cases are similar. 

For the sufficiency condition, we have four cases.

\smallskip 
\noindent{\textbf{Case 1.}} $\lambda_1 > \lambda_2$. Since $G_{\lambda/\mu}$ is connected, we deduce that $\mu_1 < \lambda_2 < \lambda_1$. We now prove that $x_1$ is a shredding vertex of $G_{\lambda/\mu}$. By assumption, $G\backslash x_1$ and $G \backslash N[x_1]$ are vertex decomposable. Let $F$ be any independent set of $G \backslash N[x_1]$. Then $F \cup \{y_m\}$ is an independent set of $G \backslash x_1$ where $m = \lambda_1$.

\smallskip 
\noindent{\textbf{Case 2.}} $\mu_1 + 1 = \lambda_1$. Then $y_{\lambda_1}$ is the shredding vertex. Indeed, by assumption, we deduce that $G\backslash y_{\lambda_1}$ and $G\backslash N[y_{\lambda_1}]$ are vertex decomposable. Let $F$ be any independent set of $G \backslash N[y_{\lambda_1}]$. Then $F \cup \{x_1\}$ is an independent set of $G \backslash y_{\lambda_1}$. 

The case $\lambda_1' > \lambda_2'$ or $\mu_1' + 1 = \lambda_1'$ follows from the previous cases by taking the conjugate. The conclusion follows.    
\end{proof}

\begin{exam} We consider the following examples to illustrate the conditions of the theorem.
\begin{enumerate}
    \item Let  $\lambda=(5,4,4)$ and $\mu=(2)$. 
    
    
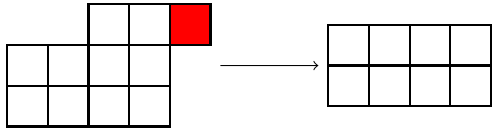
\begin{center}
     
\begin{tikzpicture}

\node (A) {
\begin{ytableau}
\none     & \none & *(white) & *(white) & *(red)  \\
     *(white)   & *(white)   & *(white)   & *(white)     \\
    *(white)    & *(white) & *(white)& *(white)
\end{ytableau}
};
\node[right of=A, xshift=3cm] (B)  {
\begin{ytableau}
     *(white)   & *(white)   & *(white)   & *(white)     \\
    *(white)    & *(white) & *(white)& *(white)
\end{ytableau}
};

\draw[->] (A) -- (B);

\end{tikzpicture}
\end{center}    
Since the second shape is non sequentially Cohen-Macaulay, we deduce that $I_{\lambda/\mu}$ is not sequentially Cohen-Macaulay. The red box can be considered as a pivot box.

    \item Let $\lambda = (5,5,4)$ and $\mu = (2,1)$. The yellow box in the middle skew shape corresponding to $\lambda/\mu$ can be considered as the pivot box. By Theorem \ref{thm_scm_Ferrers}, the shapes on both sides are sequentially Cohen-Macaulay. By Theorem \ref{thm_scm_skew_Ferrers}, we deduce that $I_{\lambda/\mu}$ is sequentially Cohen-Macaulay.
    
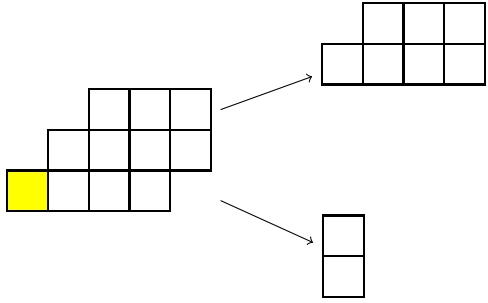
\begin{center}
	\begin{tikzpicture} 
 			\node (B)  {
 			\begin{ytableau}
 				\none     & \none & *(white) & *(white) & *(white) \\
 				\none   & *(white)  & *(white)   & *(white)  & *(white)   \\
 				*(yellow)   & *(white) & *(white) & *(white)
 			\end{ytableau}
 		};	\node[above  right of=B, node distance = 2cm, xshift=2.5cm]   (A) {
 				\begin{ytableau}
 					\none & *(white) & *(white) & *(white)  \\
 					*(white) & *(white)   & *(white)   & *(white)    
 				\end{ytableau}
 			}; 
 				\node[below right of=B,  node distance = 2cm, xshift=1.7cm]  (C) {
 				\begin{ytableau}
 					*(white)  \\
 					*(white)
 				\end{ytableau}
 			}; 
 			\draw[->] (B) -- (A);
 			\draw[->] (B) -- (C);
 			
 		\end{tikzpicture}   
\end{center}    
\end{enumerate}
   
\end{exam}

We now deduce a characterization of unmixed, (generalized) Cohen-Macaulay and Buchsbaum skew Ferrers ideals. We first recall the definitions.

Let $I$ be a nonzero homogeneous ideal and  $\m = (x_1,\ldots,x_n)$ be  the maximal homogeneous ideal of $S = \k[x_1,\ldots, x_n]$. Let $H^i_{\m}(S/I)$ denote the $i$-th local cohomology module of $S/I$ with respect to $\m$. 

\begin{defn} Let $I$ be a nonzero homogeneous ideal of $S$.
\begin{enumerate}
\item $I$ is called {\it Cohen-Macaulay} if $H^i_{\m} (S/I) = 0$ for all $i < \dim(S/I)$.
\item $I$ is called {\it  Buchsbaum} if the canonical map $$\mathrm{Ext}_R^i(R/\m,S/I) \to H_{\m}^{i}(R/I)$$ is surjective for all $i < \dim(S/I)$. 
\item $I$ is called {\it generalized Cohen-Macaulay} if $H^i_{\m}(S/I)$ has finite length for all $i<\dim(S/I)$.      
\item $I$ is called {\it unmixed} if $\dim(S/I) = \dim(S/P)$  for all associated primes $P$ of $S/I$. 
\end{enumerate}
\end{defn}
For a non-zero homogeneous ideal $I$, it is well known that Cohen-Macaulay $\implies$ Buchsbaum $\implies$ generalized Cohen-Macaulay $\implies$ unmixed. We first have 
\begin{lem}\label{lem_rad_unmixed} Let $I$ be a monomial ideal of $S$. 
\begin{enumerate}
    \item If $I$ is Buchsbaum then the radical of $I$ is Buchsbaum.
    \item Let $P$ be one of the following properties: Cohen-Macaulay, generalized Cohen-Macaulay, unmixed. Assume that $I$ has property $P$. Then all associated radicals of $I$ have property $P$.  
\end{enumerate}  
\end{lem}
\begin{proof}
    The first statement follows from \cite[Theorem 2.6]{HTT}. The second statement follows from \cite[Lemma 2.19]{JS} and \cite[Theorem 2.6]{HTT}.
\end{proof}
By the results of \cite{CST}, we have
\begin{lem}\label{lem_gCM} Let $I$ be a monomial ideal of $S=k[x_1,\ldots,x_n]$. Then $I$ is generalized Cohen-Macaulay if and only if  $I$ is unmixed and $IS[x_i^{-1}]$ is Cohen-Macaulay for all $1\le i\le n$.   
\end{lem}

We have the following remark.
\begin{rem}\label{rem_shape_separation} Let $\lambda/\mu$ be a skew shape. For an index $i \in \{1, \ldots, n\}$, we have $I_{\lambda/\mu} : x_i$ corresponds to deleting all the $j$th columns for $j = \mu_i+1$ to $\lambda_i$. In other words, it is the edge ideal of the two disjoint skew shapes, one lying to the left of the $\mu_i+1$th column and one to the right of the $\lambda_i$th column. Let $P$ be one of the following properties: sequentially Cohen-Macaulay, Cohen-Macaulay, unmixed. Assume that $I_{\lambda/\mu}$ satisfies property $P$. By Lemma \ref{lem_rad_unmixed} we deduce that for any $i$ we have $I_{\lambda^1/\mu^1}$ and $I_{\lambda^2/\mu^2}$ satisfies property $P$, where $\lambda^1/\mu^1$ is obtained by deleting all columns $\ge \mu_i + 1$ of $\lambda/\mu$ and $\lambda^2/\mu^2$ is obtained by deleting the first $\lambda_i$ columns of $\lambda/\mu$. Similarly, the skew Ferrers ideals of the skew shapes obtained by deleting all the rows $\ge \mu_j'+1$ and first $\lambda_j'$ rows of $\lambda/\mu$ satisfy property $P$.

For example, let $\lambda = (6,5,5,2,2)$ and $\mu = (3,2,1)$. The upper and lower part of $\lambda/\mu$ with $i = 2$ are colored red and blue respectively. The gray horizontal line corresponds to the second row of $\lambda/\mu$. The boxes to the left of $\lambda/\mu$ with $j = 4$ are colored yellow, the right part is empty in this case. The gray vertical line corresponds to the fourth column of $\lambda/\mu$.

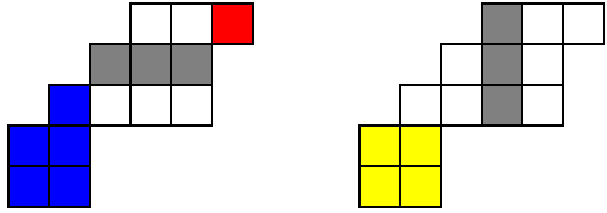
\begin{center}
\begin{tabular}{ccc}
  \begin{ytableau}
  \none   & \none   & \none   & *(white)   & *(white) & *(red) \\
  \none    & \none   & *(gray)   & *(gray)   & *(gray)      \\
  \none & *(blue) & *(white)   & *(white)   & *(white)     \\
    *(blue)   & *(blue)      \\
     *(blue)   & *(blue)      
 \end{ytableau} &   \qquad \quad & 
 \begin{ytableau}
  \none   & \none   & \none   & *(gray)   & *(white) & *(white) \\
  \none    & \none   & *(white)   & *(gray)   & *(white)      \\
  \none & *(white) & *(white)   & *(gray)   & *(white)     \\
    *(yellow)   & *(yellow)      \\
     *(yellow)   & *(yellow)      
 \end{ytableau}
\end{tabular} 
\end{center} 
\end{rem}

We will now classify skew shapes having each of these properties. First, we consider the unmixed condition. 

\begin{lem}\label{lem_unmixed_shape} Let $B_1, \ldots, B_t$ be prime unmixed shapes of alternating types and $B_i \cap B_{i+1}$ is an extremal block of both $B_i$ and $B_{i+1}$. Let $\lambda/\mu = B_1 \cup \cdots \cup B_t$. Then $I_{\lambda/\mu}$ is unmixed.    
\end{lem}
\begin{proof} We prove by induction on $t$, the number of prime unmixed shapes. The case $t = 1$ follows from the definition and \cite[Corollary 2.7]{CN}. Now assume that $t \ge 2$. Let $U$ be an extremal block of $B_1$. We assume that $B_1$ is an upper triangular prime shape of size $m$. The case where $B_1$ is a lower triangular prime shape can be done similarly. Assume that $U$ has size $p$. In other words, we have $\lambda_1 = \cdots = \lambda_p = n > \lambda_{p+1} = n-p$, $\mu_1 = \cdots = \mu_m = n-m > \mu_{m+1}$. Let $\lambda^1/\mu^1$ be the skew shape obtained by deleting the first $p$ rows of $\lambda/\mu$ and $\lambda^2/\mu^2$ be the skew shape obtained by deleting the last $m$ columns of $\lambda/\mu$. We have 
\begin{align*}
    I_{\lambda/\mu} & = I_{\lambda^1 /\mu^1} + (x_1,\ldots,x_p)(y_{n-m+1}, \ldots,y_n) \\
    & = (I_{\lambda^1 /\mu^1} + (x_1,\ldots,x_p)) \cap (I_{\lambda^1 /\mu^1} + (y_{n-m+1},\ldots,y_n))\\
    & = (I_{\lambda^1 /\mu^1} + (x_1,\ldots,x_p)) \cap (I_{\lambda^2 /\mu^2} + (y_{n-m+1},\ldots,y_n))
\end{align*}
By induction, $I_{\lambda^1 /\mu^1} + (x_1,\ldots,x_p)$ and $I_{\lambda^2 /\mu^2} + (y_{n-m+1},\ldots,y_n)$ are unmixed of the same dimension $n$. The conclusion follows.
\end{proof}

\begin{rem}
    The prime unmixed shapes in the decomposition 
    $$\lambda/\mu = B_1 \cup \cdots \cup B_t$$ 
    are unique. Furthermore, if $t > 1$ then $B_i$ are nonsquare unmixed prime shapes for all $i = 1, \ldots, t$.
\end{rem}
\begin{proof}[Proof of Theorem \ref{thm_unmixed_skew_Ferrers}] By Lemma \ref{lem_unmixed_shape}, it suffices to prove the necessary condition. When $\mu = 0$ or $\lambda = (n,n,\ldots,n)$ the conclusion follows from \cite[Corollary 2.7]{CN}. Thus, we may assume that $\mu\neq 0$, $\lambda_n < \lambda_1$, and $I_{\lambda/\mu}$ is unmixed. First, note that $(x_1, \ldots,x_n)$ and $(y_1,\ldots, y_m)$ are associated primes of $I_{\lambda/\mu}$. Hence, we must have $n = m$. We now prove by induction on $n$. The case $n \le 2$ is clear. Now, assume that $n \ge 3$. 

Let $B$ be the top-right corner block of $\lambda/\mu$. By Remark \ref{rem_shape_separation} and induction, $B$ is a square. Assume that $B$ has size $m$. By symmetry, we may assume that there are no boxes of $\lambda/\mu$ lying below $B$. For simplicity of notation, we denote by $\Delta$ the skew shape $\lambda/\mu$. We can visualize it as a block matrix $\Delta = A \cup B \cup C$, where $B$ is a square, $C$ is an unmixed shape, and $A$ is some skew shape that connects $B$ and $C$. In particular, $A$ is nonempty as we assume $G_{\lambda/\mu}$ is connected. By induction, $C = B_1 \cup \cdots \cup B_t$ where $B_i$ are prime unmixed shapes, and $B_i \cap B_{i+1}$ is an extremal block of $B_i$ and $B_{i+1}$ for all $i = 1, \ldots, t-1$.

\begin{center}
    
\begin{tikzpicture}
\draw (0,0) rectangle (2,1); 
\draw (2,0) rectangle (3,1); 
\draw (0,-2) rectangle (2,0); 

\node at (1,0.5) {$A$}; 
\node at (2.5,0.5) {$B$}; 
\node at (1,-1) {$C$}; 
\end{tikzpicture}

\end{center}

There are three cases.

\smallskip 
\noindent{\textbf{Case 1.}} $C$ is a square. In particular, $C$ has size $n- m$. Let $a$ be the largest index such that $\mu_a > 0$. Since we assume that $\mu \neq 0$, we deduce that $1 \le a < m$. Since $G_{\lambda/\mu}$ is connected, $\mu_a < n-m$. By Remark \ref{rem_shape_separation}, we deduce that the shape obtained by removing all the $j$th columns with $j \ge \mu_a + 1$ is unmixed. But this shape is a rectangle of shape $\mu_a \times (n-m)$, which is a contradiction.

\smallskip 
\noindent{\textbf{Case 2.}} $B_1$ is a nonsquare upper triangular prime unmixed shape. Assume that $B_1$ has size $p$. In particular, we have $\lambda_{m+p+1} < \lambda_{m+p} = \cdots = \lambda_{m+1} = n - m$ and $\mu_{m+p} = n - p-m$. In particular, $\mu_1 \ge n - p - m$. Assume by contradiction that $\mu_1 > n - (p+m)$. By Lemma \ref{lem_rad_unmixed}, we deduce that the shape obtained by deleting all the $j$th columns with $j > \mu_1$ is unmixed. Since $\lambda_{m+1} = n - m > \mu_1$, we deduce that all the rows from $m+1$ to $n$ survive. In other words, the remaining shape has $n - m$ rows but only $n - \mu_1 < n-m$ columns. This is a contradiction. Hence, we must have $\mu_1 = \ldots = \mu_m = n - (m+p)$. In other words, the first horizontal strip containing $B$ of $\Delta$ union with $B_1$ is a prime unmixed shape. Hence, $\Delta$ has a decomposition $\Delta = B_1' \cup B_2 \cup \cdots \cup B_t$ with $B_1'$ is upper triangular prime unmixed shape.

\smallskip 
\noindent{\textbf{Case 3.}} $B_1$ is a nonsquare lower triangular prime unmixed shape. Let $D$ be the top right extremal block of $B_1$. Assume that $D$ has size $p$. In particular, we have 
\begin{align*}
    \mu_{m+1} & = \cdots = \mu_{m+p} = n - m -p > \mu_{m+p+1}\\
    \lambda_{m+1} & \cdots = \lambda_{m+p} = n - m.
\end{align*}
Hence, $\mu_1 \ge \mu_{m+1} = n -m -p$. With an argument similar to Case 2, we deduce that $\mu_1 = n - m- p$. In other words, the first $m$ rows of $\lambda/\mu$ union with $D$ is an upper triangular prime unmixed shape and $D$ is one of its extremal block. We denote this prime shape by $B_0$. Then we have 
$$\Delta = B_0 \cup B_1 \cup \cdots \cup B_t$$
where $B_i$ are prime unmixed shapes, $B_i$, $B_{i+1}$ are of alternating type and $B_i \cap B_{i+1}$ is an extremal block of $B_i$ and $B_{i+1}$.

That concludes the proof of the theorem.
\end{proof}

We now deduce a characterization of Cohen-Macaulay skew Ferrers ideals and then classify generalized Cohen-Macaulay and Buchsbaum skew Ferrers ideals.

\begin{thm} \label{CMskewYoung}  Let $\lambda/\mu$ be a skew shape such that $G_{\lambda/\mu}$ is connected. Then $I_{\lambda/\mu}$ is Cohen-Macaulay if and only if  $\lambda/\mu$ has a canonical decomposition 
$$ \lambda/\mu = B_1 \cup \cdots \cup B_t,$$
where each $B_i$ is a prime Cohen-Macaulay shape of alternating type and $B_i \cap B_{i+1}$ is an extremal box of $B_i$ and $B_{i+1}$. 
\end{thm}
\begin{proof} By Theorem \ref{thm_unmixed_skew_Ferrers} all corner boxes of $\lambda/\mu$ are square. By Theorem \ref{thm_scm_skew_Ferrers}, all these squares must have size $1$. The conclusion follows.
\end{proof}

\begin{thm} \label{gCMskewYoung}  Let $\lambda/\mu$ be a skew shape such that $G_{\lambda/\mu}$ is connected. The following conditions are equivalent: 
\begin{enumerate}
    \item  $G_{\lambda/\mu}$ is Buchsbaum;
    \item  $G_{\lambda/\mu}$ is generalized Cohen-Macaulay;
    \item either $G_{\lambda/\mu}$ is Cohen-Macaulay or $\lambda = (n,\ldots,n)$ and $\mu = (0,\ldots,0)$. 
\end{enumerate} 
\end{thm}
\begin{proof} Clearly (1) implies (2). We now prove (2) implies (3). Since $G_{\lambda/\mu}$ is generalized Cohen-Macaulay, it is unmixed. By Theorem \ref{thm_unmixed_skew_Ferrers}, we have a decomposition $\lambda/\mu = B_1 \cup B_2 \cup \cdots \cup B_t$ of $\lambda/\mu$ into union of prime unmixed shapes $B_i$. By Theorem \ref{CMskewYoung}, if $B_i$ are Cohen-Macaulay for all $i$ then $G_{\lambda/\mu}$ is Cohen-Macaulay. Assume by contradiction that $\lambda/\mu$ is not Cohen-Macaulay and has at least $2$ corner blocks. Let $U$ be the top-right corner block of $B_1$. Taking the conjugate if necessary, we may assume that after removing $U$, we still have a block of size at least $2$. Now, let $(1,n)$ be a box in the first block. If $B_1$ is upper triangular, we consider $G \backslash N[y_n]$; if $B_1$ is lower triangular, we consider $G\backslash N[x_1]$. This corresponds to the shape obtained by removing all blocks in the same rows or same columns with the first block. By Lemma \ref{lem_gCM} the remaining shape is Cohen-Macaulay. This is a contradiction to Theorem \ref{CMskewYoung}.

Finally, we prove (3) implies (1). It remains to consider the case $\lambda = (n,\ldots,n)$ and $\mu = 0$. This follows immediately from the definition, as $H^i_\m(I_\lambda) = 0$ for all $0 < i < \dim (S/I_\lambda)$ and $H^0_\m(I_{\lambda}) \cong \k$. 
\end{proof}

\section{Skew tableau ideals}\label{sec_skew_tableau} 
In this section, we classify all sequentially Cohen-Macaulay skew tableau ideals and then all (generalized) Cohen-Macaulay and Buchsbaum skew tableau ideals. We first recall the definition of edge ideals of edge weighted graphs introduced by Paulsen and Sather-Wagstaff \cite{PS}.

Let $G$ denote a finite simple graph over the vertex set $V(G) = [n] = \{1,\ldots,n\}$ and the edge set $E(G)$. Let $\ww:E(G)\rightarrow \ZZ_{>0}$ be a weight function on edges of $G$. The edge ideal of the edge-weighted graph $(G,\ww)$ is
$$I(G,\ww) = \big( (x_i x_j)^{w(i,j)} \mid \{i,j\}\in E(G)\big) \subseteq S.$$
From the proof of \cite[Theorem 1.1]{DMV} we have the following

\begin{thm}\label{gCMallweights}
Let $G$ be a simple graph. The following statements are equivalent:
\begin{enumerate}
    \item  $I(G,\ww)$ is Cohen-Macaulay for all weight functions $\ww$.
\item  $I(G,\ww)$ is Buchsbaum for all weight functions $\ww$.
\item  $I(G,\ww)$ is generalized Cohen-Macaulay for all weight functions $\ww$.
\item $I(G,\ww)$ is unmixed for all weight functions $\ww$.
\item $G$ is the disjoint union of finitely many complete graphs.
\end{enumerate}
\end{thm}
\begin{proof} We already have $(1) \implies (2) \implies (3) \implies (4)$. For $(4) \implies (5)$, by Lemma \ref{lem_rad_unmixed} and the proof of \cite[Theorem 1.1]{DMV}, we deduce that all induced subgraphs of $G$ are unmixed. In particular, $P_3$ is not an induced subgraph of $G$. Hence, $G$ is the disjoint union of finitely many complete graphs. By \cite[Theorem 1.1]{DMV}, $(4) \implies (1)$. The conclusion follows.    
\end{proof}

A \textit{filling} of a skew Young diagram is an assignment of positive integers $w(i,j) $ to each box $ (i,j) $ in the skew diagram. 
\begin{defn}
    The {\it skew tableau ideal} associated with a skew Young diagram $\lambda/\mu$ with a given filling $Y$ is the edge ideal of the edge-weighted skew Ferrers graph $G_{\lambda/\mu}$ with the weight function given by the filling $Y$. In other words,
$$I (Y):= ((x_iy_j)^{w(i,j)}\mid 1\le i\le n, \mu_i+1\le j \le \lambda_i).$$ 
\end{defn}
Since nontrivial skew Ferrers graphs are not complete graphs, we deduce that there exist a filling $Y$ such that $I(Y)$ is not Cohen-Macaulay. We first establish a characterization of unmixed skew tableau ideals.

\begin{proof}[Proof of Theorem \ref{thm_unmixed_skew_tableau}] We first prove the necessary condition. Assume that $I(Y)$ is unmixed. By Lemma \ref{lem_rad_unmixed}, Theorem \ref{thm_unmixed_skew_Ferrers}, and symmetry, we may assume that $\mu = 0$ and $\lambda$ is unmixed. We prove by induction on the number of corner blocks of $\lambda$. First assume that $\lambda$ is a square. Let $\omega$ be the smallest weight on the first row of $\lambda$. Assume by contradiction that there exists $i$ such that $w(1,i) > \omega$. We have $\sqrt{I(Y) : x_1^{\omega}}$ is the Ferrers ideal corresponding to deleting all columns $j$ such that $w(1,j) = \omega$. Since $w(1,i) > \omega$, this column still survive. Hence, the resulting Ferrers diagram is a nonsquare rectangle. This is a contradiction. Thus, the weight function is constant on the first row. Similarly, $Y$ is constant on the first column. This argument apply to any rows and columns of $\lambda$. Hence, $Y$ is constant. 

Now, assume that $\lambda$ has more than one corner blocks. Let $B_1, \ldots, B_t$ be the square corner blocks of $\lambda$ of size $b_1, \ldots, b_t$ respectively. By Lemma \ref{lem_rad_unmixed} and induction, we deduce that $Y$ is constant on all the blocks of $\Delta_1$ and $\Delta_2$ where $\Delta_1$ is obtained by deleting the first $b_1$ rows and $\Delta_2$ is obtained by deleting the first $b_t$ columns. In other words, $Y$ is constant on all blocks except the first block which is the intersection of the first $b_1$ rows and the first $b_t$ columns. The argument in the previous paragraph applies, as if $Y$ is not constant on this block, then we will get an unmixed shape of non-square size, which is a contradiction.

We now prove that $Y$ is weakly increasing on both rows and columns. By induction and Lemma \ref{lem_rad_unmixed}, and symmetry, it suffices to prove the weakly increasing property on the blocks of the first horizontal strip. Assume by contradiction that $Y$ is not weakly increasing on the first horizontal strip. As in the proof of \cite[Corollary 3.7]{HV}, we let $j$ be the smallest index such that $w(1,j) > w(1,j+1)$. We have 
$$\sqrt{I(Y) : x^{w(1,j+1)}} = I_\lambda + (y_{j+1}, y_\ell \mid \ell \in V)$$
for some $V \subseteq \{j+2,\ldots,n\}$. The latter ideal corresponds to the deletion of the $j+1$th column and $\ell$th column for $\ell \in V$. The remaining shape is not of square type, we deduce a contradiction.

Now, we prove the sufficiency condition. We assume that $\lambda/\mu = B_1 \cup \cdots \cup B_t$ is the canonical decomposition of prime unmixed shapes and $Y$ is constant on each block of $B_j$ and $Y$ is weakly increasing on triangular prime unmixed shapes and weakly decreasing on triangular prime unmixed shapes. We consider the case where $B_1$ is an upper triangular prime unmixed shape as in the proof of Lemma \ref{lem_unmixed_shape}. The other case can be done in similar manner. Let $U$ be an extremal block of $B_1$. Assume that $B_1$ has size $m$, $U$ has size $p$ and the constant value of $Y$ on $U$ is $\omega$. In particular, we have $\lambda_1 = \cdots = \lambda_p = n > \lambda_{p+1} = n - p$, $\mu_1 = \cdots = \mu_{m} = n -m$. Let $Y_1$ be the skew tableau obtained by deleting the first $p$ rows of $Y$ and $Y_2$ be the skew tableau obtained by deleting the lat $m$ columns of $Y$. We have 
\begin{align*}
    I(Y) &= I(Y_1) + (x_1^\omega, \ldots, x_p^{\omega})(y_{n-m+1}^\omega, \ldots, y_{n}^\omega)\\
    & = (I(Y_1) + (x_1^\omega, \ldots, x_p^\omega))  \cap (I(Y_1) + (y_{n-m+1}^\omega, \ldots, y_{n}^\omega)).
\end{align*}
Since $Y$ is weakly increasing on $B_1$, we deduce that $(I(Y_1) + (y_{n-m+1}^\omega, \ldots, y_{n}^\omega)) = I(Y_2) +  (y_{n-m+1}^\omega, \ldots, y_{n}^\omega)$. By induction, 
$$(I(Y_1) + (x_1^\omega, \ldots, x_p^\omega)) \text{ and } I(Y_2) +  (y_{n-m+1}^\omega, \ldots, y_{n}^\omega)$$ 
are unmixed of dimension $n$. The conclusion follows.
\end{proof}

\begin{exam} The following example give an unmixed and a mixed tableau ideals on the same unmixed shape. The red and yellow blocks are the connecting blocks.
    
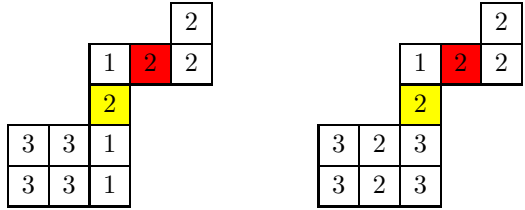
\begin{center}
\begin{tabular}{ccc}
  \begin{ytableau}
       \none & \none & \none & \none & *(white)2  \\
     \none    & \none & *(white)1 & *(red)2 & *(white)2  \\
      \none    &  \none  & *(yellow)2  \\
    *(white)3    & *(white)3 & *(white)1   \\
    *(white)3    & *(white)3  & *(white)1 
   \end{ytableau} &   \qquad \quad & 
 \begin{ytableau}
       \none & \none & \none & \none & *(white)2  \\
     \none    & \none & *(white)1 & *(red)2 & *(white)2  \\
      \none    &  \none  & *(yellow)2  \\
    *(white)3    & *(white)2 & *(white)3   \\
    *(white)3    & *(white)2  & *(white)3 
   \end{ytableau}
\end{tabular} 
\end{center} 

\end{exam}
We now deduce a generalization of \cite[Corollary 3.7]{HV} for skew tableau ideals and then classify generalized Cohen-Macaulay and Buchsbaum skew tableau ideals.

\begin{thm} \label{thm_CM_skewtableau}  Let $Y$ be an arbitrary filling of a skew Young diagram $\lambda/\mu$.  The following conditions are equivalent: 
\begin{enumerate}
    \item $I(Y)$ is Cohen--Macaulay; 
    \item $\lambda/\mu$ is Cohen-Macaulay with a decomposition $\lambda/\mu = B_1 \cup  \cdots \cup B_t$; $Y$ is weakly increasing on upper triangular prime shapes and weakly decreasing on lower triangular prime shapes.
\end{enumerate} 
\end{thm} 
\begin{proof} The conclusion follows from Theorem \ref{CMskewYoung} and Theorem \ref{thm_unmixed_skew_tableau}.
\end{proof}

\begin{thm}\label{thm_gCM_skew_tableau}   Let $Y$ be an arbitrary filling of a skew  Young diagram $\lambda/\mu$.  The following conditions are equivalent: 
\begin{enumerate}
    \item $I(Y)$ is Buchsbaum; 
    \item $I(Y)$ is generalized Cohen--Macaulay; 
    \item either $I(Y)$ is Cohen-Macaulay or $\lambda = (n,\ldots,n)$ and $\mu = 0$ and $Y$ is constant.
    \end{enumerate}
\end{thm}
\begin{proof} That (1) implies (2) is well-known. By Theorem \ref{thm_unmixed_skew_tableau}, Theorem \ref{gCMskewYoung} and the fact that generalized Cohen-Macaulay ideals are unmixed, we deduce that (2) implies (3).

We now prove that (3) implies (1). It remains to consider the case $\lambda = (n,\ldots,n)$, $\mu = 0$ and $Y$ is constant. The conclusion follows from \cite[Theorem 3.3]{MN}.
\end{proof}

To establish the sequentially Cohen-Macaulay property of skew tableau ideals, we need to the following results, \cite[Lemma 2.13, Lemma 2.15]{HV} about the associated radicals of edge ideals of edge weighted graphs.
\begin{lem}\label{lem_associated_radicals_edge_weight} Let $G$ be a simple graph on the vertex set $V(G) = \{1,\ldots,n\}$ and $\ww: E(G) \to \ZZ_+$ be a weight function. For any exponent $\a \in \NN^n$, let 
$$U=\{i \mid \text{there exists } j \text{ such that } \{i,j\}\in E(G) \text{ and }  a_i <  w(i,j) \le a_j\}.$$ 
Then  
$$\sqrt{I(G,{\ww}) : x^\a} = I(G\backslash U) + (x_i \mid i \in U),$$
where $I(G\backslash U)$ is the edge ideal of the induced subgraph of $G$ on $V(G) \backslash U$.
\end{lem}

\begin{lem}\label{lem_associated_rad_2} Let $(G,{\ww})$ be an edge-weighted graph. Assume that $\sqrt{I(G,{\ww}):x^\a} = I(G,{\ww}) + (x_i \mid i \in U)$ for some $U \subseteq [n]$. Let $I(G',{\ww}') = I(G,{\ww}) + (x_i \mid i \in U)$. Then an associated radical of $I(G',{\ww}')$ is also an associated radical of $I(G,{\ww})$.  
\end{lem}

\begin{lem}\label{lem_non_essential_vars} Let $Y$ be a filling of positive integers in a skew shape $\lambda/\mu$. Assume that $I(Y)$ is sequentially Cohen-Macaulay and $\lambda_1 = \lambda_2$. Then for any $j$ such that $\mu_1+1 \le j \le \lambda_1$ we have $I(Y) + (y_j)$ is sequentially Cohen-Macaulay.    
\end{lem}
\begin{proof} Let $J = I(Y) + (y_j)$. By Lemma \ref{lem_scm_radical}, we need to prove that all associated radicals of $J$ are sequentially Cohen-Macaulay. Let $K$ be an associated radical of $J$. Then either $K$ is an associated radical of $I(Y)$ or $K = L + (y_j)$ where $L$ is an associated radical of $I(Y)$ corresponding to some skew shape $\lambda^1/\mu^1$ obtained by removing some rows and columns of $\lambda/\mu$. Since $K \neq L$, we deduce that $\lambda^1_1 = \lambda^1_2$. Thus, it suffices to prove the claim in the case $I(Y) = I_{\lambda/\mu}$. The conclusion then follows from Theorem \ref{thm_scm_skew_Ferrers} and induction.    
\end{proof}

\begin{proof}[Proof of Theorem \ref{thm_scm_skew_tableau}] By Lemma \ref{lem_scm_radical} and Lemma \ref{lem_shape_condition}, we deduce that either $\lambda_1 > \lambda_2$ or $\mu_1 + 1 = \lambda_1$ or $\lambda_1' > \lambda_2'$ or $\mu_1' + 1 = \lambda_1'$. We consider the case $\lambda_1 > \lambda_2$ only. The other cases can be done similarly and are left as an exercise for interested readers. We first assume that $I(Y)$ is sequentially Cohen-Macaulay and prove the the necessary condition. We have 
$$I(Y) : x_1^{\omega_1} = I(Y_2) + (y_{\ell} \mid w(1,\ell) \le \omega_1) \text{ and } I(Y): y_{m}^{\omega_1} = I(Y_1) + (x_1)$$
where $\omega_1 = \max \{w(1,j) \mid j = \lambda_1 + 1, \ldots, m\}$. By Lemma \ref{lem_associated_rad_2}, associated radicals of $I(Y_2) + (y_{\ell} \mid w(1,\ell) \le \omega_1)$ and $I(Y_1) + (x_1)$ are associated radicals of $I(Y)$. By Lemma \ref{lem_scm_radical}, we deduce that $I(Y_1)$ and $I(Y_2)$ are sequentially Cohen-Macaulay.

We now assume that $\lambda_1 > \lambda_2$, $Y_1$ and $Y_2$ are sequentially Cohen-Macaulay. We need to prove that $Y$ is sequentially Cohen-Macaulay. By Lemma \ref{lem_scm_radical}, it suffices to prove that all associated radicals of $I(Y)$ are sequentially Cohen-Macaulay. Let 
$$J = \sqrt{I(Y) : x^\a y^\b} = I_{\lambda/\mu} + (x_i, y_j \mid i \in U, j \in V)$$ 
for some $U \subseteq \{1,\ldots,n\}$ and $V \subseteq \{1,\ldots,m\}$ be an associated radical of $I(Y)$. There are two cases:

\smallskip
\noindent {\bf Case 1.} $1 \notin U$. If $j \in V$ for all $j \ge \lambda_2+1$ we deduce that $a_1 \ge \omega_1$. Let $\mathbf{c}, \mathbf{d}$ be the restriction of $\a,\b$ to the support of $Y_1$. By Lemma \ref{lem_associated_radicals_edge_weight}, we have $\sqrt{I(Y_1) : x^{\mathbf{c}} y^{\mathbf{d}} } + (y_\ell \mid w(1,\ell) \le \omega_1) = \sqrt{I(Y) : x^\a y^\b} = J$. Since $Y_1$ is sequentially Cohen-Macaulay, we deduce that $J$ is sequentially Cohen-Macaulay. Now, assume that $j\notin V$ for some $j \ge \lambda_2 + 1$. By Theorem \ref{thm_scm_skew_Ferrers}, it suffices to prove that $J + (x_1)$ and $J + (y_{\mu_1+1},\ldots, y_{\lambda_1})$ are sequentially Cohen-Macaulay. By Lemma \ref{lem_associated_rad_2}, they are associated radicals of $Y_1$ and $Y_2$ respectively. The conclusion follows.

\smallskip
\noindent {\bf Case 2.} $1 \in U$. If $a_1 < \min (w(1,j) \mid j = \mu_1 + 1, \ldots,\lambda_2)$ then $J$ is an associated radical of $Y_1$. Thus, we may assume that $a_1 \ge w(1,j)$ for some $j \in \{\mu_1+1, \ldots, \lambda_2\}$. In this case $J$ is an associated radical of $Y_2 + (y_\ell \mid \ell \in V)$ for some $V$ such that $\{\lambda_2+1, \ldots, \lambda_1\} \subseteq V$. The conclusion follows from Lemma \ref{lem_non_essential_vars}. 
\end{proof}
 \section*{Acknowledgments}
We thank Professor Uwe Nagel and Professor Nguyen Cong Minh for their many useful discussions and suggestions. A part of this work was completed during the first author’s visit to the Vietnam Institute for Advanced Study in Mathematics (VIASM). He wishes to express his gratitude to VIASM for its hospitality and financial support.


\begin{thebibliography}{2}


\bibitem[CN]{CN} A. Corso and U. Nagel, \emph{Monomial and toric ideals associated to Young graphs},  Trans. Amer. Math. Soc. {\bf 361}(3), (2005) 591--613.

\bibitem[CST]{CST}
 N. T. Cuong, P. Schenzel and N. V. Trung, {\it Multiplizitaten in verallgemeinerten Cohen-Macaulay Moduln}, Math. Nachr. {\bf 85} (1978), 57--73.
 
\bibitem[DMV]{DMV} L. T. K. Diem, N. C. Minh, and T. Vu, {\it The sequentially Cohen-Macaulay property of edge ideals of edge-weighted graphs}, J. Algebraic Combin. {\bf 60} (2024), 589--597.

\bibitem[dAH]{dAH}  H. de Alba and D. T. Hoang, {\it   On the extremal Betti numbers of binomial edge ideals of closed graphs},    Math. Nachr. {\bf 291} (2018) 28--40.

 
\bibitem[Ful]{Ful}
W. Fulton, 
{\em Young tableaux}, London Mathematical Society Student Texts {\bf 35} (1996).



 
\bibitem[HTT] {HTT} 
J. Herzog, Y. Takayama, and N. Terai,
{\em On the radical of a monomial ideal,} Arch. Math. {\bf 85} (2005), 397--408.

\bibitem[H]{H} D. T. Hoang, {\it On the Betti numbers of edge ideal of skew Ferrers graphs}, Int. J. Algebra Comput. {\bf 30} (1) (2020), 125--139.


\bibitem[HV]{HV} D. T.  Hoang and T.  Vu, {\it  Depth and regularity of tableau ideals}, submitted. 

 
\bibitem[JS]{JS} R. Jafari, and  H. Sabzrou,  {\it Associated radical ideals of monomial ideals},  Comm. Algebra {\bf 47} (3) (2019), 1029--1042.
 


\bibitem[MN]{MN}
N. C. Minh and Y. Nakamura,
{\em Buchsbaumness of ordinary powers of two-dimensional square-free monomial ideals}, J. Algebra {\bf 327} (2011), 292--306.

	
\bibitem[P]{P} Irena Peeva, {\it Closed binomial edge ideals}, J. Reine Angew. Math. {\bf 803} (2023), 1--33.

\bibitem[PS] {PS} 
C. Paulsen and S. Sather-Wagstaff,
{\em Edge ideals of weighted graphs}, 
J. Algebra Appl. {\bf 12} (2013), no. 5, 1250223.

\bibitem[S] {S} R. Stanley, {\it Combinatorics and Commutative Algebra}, 2. Edition, Birkh\"auser, 1996. 
 

\bibitem[V]{V} A. Van Tuyl, {\it  Sequentially Cohen–Macaulay bipartite graphs: vertex decomposability and regularity},   Arch. Math. {\bf 93} (2009), 451--459. 


\bibitem[W] {W}
R. Woodroofe,
{\it Vertex decomposable graphs and obstructions to shellability},
Proc. Amer. Math. Soc. {\bf 137} (2009), no. 10, 3235--3246.


\end{thebibliography}
\end{document}